\documentclass[reqno, 12pt]{amsart}
\usepackage{amsmath,mathtools}
 \usepackage{amssymb}
\usepackage{amsthm}
\usepackage{times}
\usepackage{latexsym}
\usepackage[mathscr]{eucal}

\numberwithin{equation}{section}
 
  \newtheorem{theorem}{Theorem}[section]
  \newtheorem{proposition}[theorem]{Proposition}
  \newtheorem{lemma}[theorem]{Lemma}
  \newtheorem{corollary}[theorem]{Corollary}

  \newtheorem{definition}[theorem]{Definition}
  \newtheorem{example}[theorem]{Example}

\title[Some remarks on quasi generalized CR-null geometry]{Some remarks on quasi generalized CR-null geometry in indefinite nearly cosymplectic manifolds}

\author[Fortun\'{e} Massamba, Samuel Ssekajja ]{Fortun\'{e} Massamba*, Samuel Ssekajja**}

\newcommand{\acr}{\newline\indent}

\address{\llap{*\,} School of Mathematics, Statistics and Computer Science\acr
 University of KwaZulu-Natal\acr
 Private Bag X01, Scottsville 3209\acr
South Africa}
\email{massfort@yahoo.fr, Massamba@ukzn.ac.za} 
\thanks{}

\address{\llap{**\,} School of Mathematics, Statistics and Computer Science\acr
 University of KwaZulu-Natal\acr
 Private Bag X01, Scottsville 3209\acr
South Africa}
\email{ssekajja.samuel.buwaga@aims-senegal.org} 
\thanks{}



\begin{document}
 
\begin{abstract}
In \cite{ms}, the authors initiated the study of quasi generalized CR (QGCR)-null submanifolds. In this paper, attention is drawn to some distributions on ascreen QGCR-null submanifolds in an indefinite nearly cosymplectic manifold. We characterize totally umbilical and irrotational  ascreen QGCR-null submanifolds. We finally discuss the geometric effects of geodesity conditions on such submanifold.
\end{abstract} 

\maketitle

\section{Introduction}   
 One of the current interesting research areas in semi-Riemannian geometry is the theory of null (or lighlike) submanifolds. An intrinsic approach to the theory of null submanifols was advanced by D. N. Kupeli \cite{kup}, yet an extrinsic counterpart had to wait for Duggal-Bejancu \cite{db}, and later by Duggal-Sahin \cite{ds2}. Since then, many researchers have labored to extend their theories with evidence from these few selected papers: \cite{dj}, \cite{ds2}, \cite{ds3}, \cite{ds4}, \cite{dh}, \cite{ma1}, \cite{ma2}, \cite{ma22} and other references therein. The rapid increase in research on this topic, since 1996, is inspired by the numerous applications of the theory to  mathematical physics, particularly in general relativity. More precisely, in general relativity,  null submanifolds represent different models of black hole horizons (see \cite{db} and \cite{ds2} for details).
 
 In \cite{ds3}, the authors initiated the study of generalized CR (GCR)-null submanifolds of an indefinite Sasakian manifold, which are \textit{tangent} to the structure vector field, $\xi$, of the almost contact structure $(\overline{\phi}, \xi, \eta)$. Moreover, when $\xi$ is tangent to the submanifold, C. Calin \cite{ca} proved that it belongs to its screen distribution. This assumption is widely accepted and it has been applied in many papers on null contact geometry, for instance \cite{ds1}, \cite{ds2}, \cite{ds3}, \cite{ma1}, \cite{ma2} and \cite{ma22}. It is worthy mentioning that $\xi$ is a \textit{global vector field} defined on the entire tangent bundle of the ambient almost contact manifold. Thus, restricting it to  the screen distribution is only one of those cases in which it can be placed. In the study of Riemannian CR-submanifolds of Sasakian manifolds, Yano-Kon \cite[page 48]{yano} proved that making $\xi$ a normal vector field in such scenario leads to an anti-invariant submanifold, and hence $\xi$ was kept tangent to the CR-submanifold. Their proof leans against the fact that; the shape operator on such CR-submanifold is naturally \textit{symmetric}  with respect to the induced Riemannian metric $g$. On the other hand, the shape operators of any $r$-null submanifold are generally \textit{not} symmetric with respect to the induced degenerate metric $g$ (see \cite{db} and \cite{ds2} for details).
 
 In an attempt to generalize $\xi$,  we introduced a special class of CR-null submanifold of a nearly Sasakian manifold, known as \textit{quasi generalized CR (QGCR)}-null submanifold \cite{ms}, for which the classical GCR-null submanifolds \cite{ds2} forms part. Among other benefits, generalizing $\xi$ leads to QGCR-null submanifolds of lower dimensions and with quite different geometric properties compared to respective GCR-null submanifolds.
 
The purpose of this paper is to investigate the geometry of distributions on ascreen QGCR-null submanifolds of indefinite nearly cosymplectic manifolds. The paper is organized as follows; In Section \ref{Preli}, we present the basic notions of null submanifolds and nearly cosymplectic manifolds. More details can be found in \cite{mo}, \cite{bl1}, \cite{bl2}, \cite{be}, \cite{endo} and \cite{diaz}. In Section \ref{AlmostGe}, we review the basic notions of QGCR-null submanifolds and we give an example of ascreen QGCR-null submanifold. In Section \ref{existence}, we discuss totally umbilical, totally geodesic and irrotational ascreen QGCR-null submanifolds of an indefinte nearly cosymplectic space form $\overline{M}(\overline{c})$. Finally, in Section \ref{mixed} we investigate the geodesity of the distributions $D$ and $\widehat{D}$.

\section{Preliminaries}\label{Preli}

Let  $M^{m}$ be a codimension $n$  submanifold of a semi Riemannian manifold $(\overline{M},\overline{g})$ of constant index $\nu$, $1\le \nu\le m+n$, where both $m,n\ge 1$. Then, $M$ is said to be a \textit{null submanifold} of $\overline{M}$ if the tangent and normal bundles of $M$ have a non-trivial intersection. This intersection defines a smooth distribution  on $M$, called the \textit{radical} distribution \cite{db}. More precisely, consider $p\in M$, one  defines the orthogonal complement $T_{p} M^{\perp}$ of the tangent space $T_{p} M$ by
 $$
 T_{p} M^{\perp} = \{X\in T_{p} M: \overline{g}(X, Y)=0,\; \forall \, Y\in T_{p} M)\}.
 $$
 If we denote the radical distribution on $M$ by $\mathrm{Rad} \, T_{p} M$, then 
  $\mathrm{Rad} \, T_{p} M = \mathrm{Rad}\, T_{p} M^{\perp} = T_{p} M \cap T_{p} M^{\perp}$. The submanifold $M$ of $\overline{M}$ is said to be $r$-\textit{null submanifold} (one supposes that the index of $\overline{M}$ is $\nu \ge r$), if the mapping 
 $ 
 \mathrm{Rad} \, T M: p\in M \longrightarrow\mathrm{Rad}\, T_{p} M 
 $ 
 defines a \textit{smooth distribution} on $M$ of rank $r > 0$. 
 
 In this paper, an $r$-null submanifold will simply be called a \textit{null submanifold} and $g=\overline{g}|_{TM}$ is a \textit{null metric}, unless we need to specify $r$.
 
 Let $S(T M)$ be a screen distribution which is a semi-Riemannian complementary distribution of $\mathrm{Rad}\,T M$ in $T M$, that is,
 \begin{equation}\label{eq05}
  T M = \mathrm{Rad}\,T M \perp S(T M).
 \end{equation} 
Consider a screen transversal bundle $S(TM^\perp)$, which is semi-Riemannian and complementary to $\mathrm{Rad}\, TM$ in $TM^\perp$. For any local basis $\{E_{1},\cdots,E_{r}\}$ of  $\mathrm{Rad}\,TM$, there exists a local null frame $\{N_{1},\cdots,	N_{r}\}\subset S(T M^\perp)$ in $S(T M )^\perp$  such that $g(E_i , N_j ) = \delta_{ij}$ and $\overline{g}(N_{i},N_{j})=0$. It follows that there exists a null transversal vector bundle $l\mathrm{tr}(TM)$ locally spanned by $\{N_{1},\cdots,N_{r}\}$ (see details in \cite{db} and \cite{ds2}). If $\mathrm{tr}(TM)$ denotes the complementary (but not orthogonal) vector bundle to $TM$ in $T\overline{M}$. Then, 
\begin{align}\label{eq04}
           &\mathrm{tr}(TM)=l\mathrm{tr}(TM)\perp S(TM^\perp),\\\label{eq08}
  &T\overline{M}= S(TM)\perp S(TM^\perp)\perp\{\mathrm{Rad}\, TM\oplus l\mathrm{tr}(TM)\} .
\end{align}
It is important to note that the screen distribution $S(TM)$ is not unique, and is canonically isomorphic to the factor vector bundle $TM/ \mathrm{Rad}\, TM$ \cite{kup}.
 
 Given a null submanifold $M$, then the following classifications of $M$ are well-known \cite{db}: i). $M$ is $r$-null if $1\leq r< min\{m,n\}$; ii). $M$ is co-isotropic if $1\leq r=n<m$,  $S(TM^\perp)=\{0\}$; iii). $M$ is isotropic if $1\leq r=m<n$,  $S(TM)=\{0\}$; iv). $M$ is totally null if $r=n=m$,  $S(TM)=S(TM^\perp)=\{0\}$. 
 
 Where necessary, the following range of indices will be used;
\begin{equation*}
 i,j,k\in\{1,\ldots, r\}, \hspace{.2cm}\alpha,\beta,\gamma\in\{r+1,\ldots, n\}.
\end{equation*}
Consider a local quasi-orthonormal fields of frames of $\overline{M}$ along $M$ as 
\begin{equation*}
\{ E_1,\cdots, E_r,N_1,\cdots, N_r,X_{r+1},\cdots,X_{m},W_{1+r},\cdots, W_{n}\},
\end{equation*}
where
 $\{X_{r+1},\cdots,X_{m}\}$ and $\{W_{1+r},\ldots, W_n\}$ are respectively orthonormal bases of $\Gamma(S(TM)|_{U})$ and $\Gamma(S(TM^{\perp})|_{U})$.
 
Throughout the paper we consider $\Gamma(\Xi)$  to be  a set of smooth sections of the vector bundle $\Xi$.

Let $P$ be the projection morphism of $TM$ on to $S(TM)$. Then,  the Gauss-Weingartein equations of an $r$-null submanifold $M$ and $S(TM)$ are the following (see \cite{db} and \cite{ds2} for detailed explanations);
  \begin{align}
    & \overline{\nabla}_X Y=\nabla_X Y+\sum_{i=1}^r h_i^l(X,Y)N_i+\sum_{\alpha=r+1}^n h_\alpha^s(X,Y)W_\alpha,\label{eq11}  \\
     & \overline{\nabla}_X N_i=-A_{N_i} X+\sum_{j=1}^r \tau_{ij}(X) N_j+\sum_{\alpha=r+1}^n \rho_{i\alpha}(X)W_\alpha,\label{eq31}
     \end{align}
     \begin{align} 
     & \overline{\nabla}_X W_\alpha=-A_{W_\alpha} X+\sum_{i=1}^r \varphi_{\alpha i}(X) N_i+\sum_{\beta=r+1}^n \sigma_{\alpha\beta}(X)W_\beta,\label{eq32}\\ 
     & \nabla_X P Y=\nabla_X^* PY+\sum_{i=1}^r h_i^*(X, P Y)E_i,\\
     & \nabla_X E_i=-A_{E_i}^* X-\sum_{j=1}^r \tau_{ji}(X) E_j,\;\;\;\; \forall\; X,Y\in \Gamma(TM)\label{eq50},
  \end{align}
where  $\nabla$ and $\nabla^*$ are the induced connections on $TM$ and $S(TM)$ respectively, $h_i^l$ and $h_\alpha^s$ are symmetric bilinear forms known as \textit{local null} and \textit{screen fundamental} forms of $TM$ respectively. Furthermore, $h_i^*$ are the \textit{second fundamental forms} of $S(TM)$. $A_{N_i}$, $A_{E_i}^*$ and $A_{W_\alpha}$ are linear operators on $TM$ while $\tau_{ij}$, $\rho_{i\alpha}$, $\varphi_{\alpha i}$ and $\sigma_{\alpha\beta}$ are 1-forms on $TM$. Note that the \textit{second fundamental tensor} of $M$ is given by
 \begin{equation}\label{h1}
 h(X,Y)=\sum_{i=1}^{r} h_{i}^{l}(X,Y)N_{i}+\sum_{\alpha=r+1}^n h_{\alpha}^{s}(X,Y)W_{\alpha},
\end{equation}
for any $X,Y\in \Gamma(TM)$.  The connection  $\nabla^{*}$ is a metric connection on $S(TM)$ while $\nabla$ is generally not a metric connection and is given by
      \begin{equation*}
         (\nabla_X g)(Y,Z)=\sum_{i=1}^r\{h_i^l(X,Y)\lambda_i(Z)+h_i^l(X,Z)\lambda_i(Y)\},
      \end{equation*}
for any $X,Y\in \Gamma(TM)$ and $\lambda_i$ are 1-forms given by
          $\lambda_i(X)=\overline{g}(X,N_i)$, for all  $X\in \Gamma(TM).$ 
 By using (\ref{eq11}), (\ref{eq31}) and (\ref{eq32}), the curvature tensors $\bar{R}$, $R$ of $\bar{M}$ and $M$, respectively are related as,  for any $X,Y,Z,W\in \Gamma(TM)$,
      \begin{align}\label{s8}
                \overline{R}(X,W,Z,Y)&=\overline{g}(R(X,W)Z,Y)+\overline{g}(A_{h^l(X,Z)}W,Y)\nonumber\\
                                     &-\overline{g}(A_{h^l(W,Z)}X,Y)+\overline{g}(A_{h^s(X,Z)}W,Y)\nonumber\\
                                     &-\overline{g}(A_{h^s(W,Z)}X,Y)+\overline{g}((\nabla_X h^l)(W,Z),Y)\nonumber\\
                                     &-\overline{g}((\nabla_W h^l)(X,Z),Y)+\overline{g}(D^l(X,h^s(W,Z)),Y)\nonumber\\
                                     &-\overline{g}(D^l(W,h^s(X,Z)),Y)+\overline{g}((\nabla_X h^s)(W,Z),Y)\nonumber\\
                                     &-\overline{g}((\nabla_W h^s)(X,Z),Y)+\overline{g}(D^s(X,h^l(W,Z)),Y)\nonumber\\
                                     &-\overline{g}(D^s(W,h^l(X,Z)),Y).
      \end{align}
 A null submanifold $(M,g)$ of an indefinite manifold $(\overline{M},\overline{g})$ is said to be totally umbilical in $\overline{M}$ \cite{ds2} if there is a smooth transversal vector field $\mathcal{H}\in \Gamma(\mathrm{tr}(TM))$, called the transversal curvature vector of $M$ such that 
          \begin{equation}\label{eq17}
             h = g\otimes \mathcal{H}.
          \end{equation}
Moreover, it is easy to see that $M$ is totally umbilical in $\overline{M}$, if and only if on each coordinate neighborhood $U$ there exist  smooth vector fields $\mathscr{H}^l\in\Gamma(l\mathrm{tr}(TM))$ and $\mathscr{H}^s\in\Gamma(S(TM^\perp))$ and smooth functions $\mathscr{H}_i^l\in F(l\mathrm{tr}(TM))$ and $\mathscr{H}_\alpha^s\in F(S(TM^\perp))$ such that,  
  \begin{align}
   h^l(X,Y) & =\mathscr{H}^l g(X,Y),\;\;\; h^s(X,Y)=\mathscr{H}^s g(X,Y),\nonumber\\
   h_i^l(X,Y) & =\mathscr{H}_i^l g(X,Y),\;\;\;h_\alpha^s(X,Y)=\mathscr{H}_\alpha^s g(X,Y)\label{s191}, 
  \end{align} 
for all $X,Y\in\Gamma(TM)$.  

Let now consider $\overline{M}$ to be a $(2n + 1)$-dimensional manifold endowed with an almost contact structure $(\overline{\phi}, \xi, \eta)$, i.e. $\overline{\phi}$ is a tensor field of type $(1, 1)$, $\xi$ is a vector field, and $\eta$ is a 1-form satisfying
\begin{equation}\label{equa1}
\overline{\phi}^{2} = -\mathbb{I} + \eta
\otimes\xi,\;\;\eta(\xi)= 1 ,\;\;\eta\circ\overline{\phi} =
0 \;\;\mbox{and}\;\;\overline{\phi}(\xi) = 0.
\end{equation}
Then $(\overline{\phi}, \xi, \eta,\,\overline{g})$ is called an \textit{indefinite almost contact metric structure} on $\overline{M}$ if $(\overline{\phi}, \xi, \eta)$ is an almost contact structure on $\overline{M}$ and $\overline{g}$ is a semi-Riemannian metric on $\overline{M}$ such that \cite{bl2}, for any vector field $\overline{X}$, $\overline{Y}$ on  $\overline{M}$,
\begin{equation}\label{equa2}
 \overline{g}(\overline{\phi}\,\overline{X}, \overline{\phi}\,\overline{Y}) = \overline{g}(\overline{X}, \overline{Y}) -  \eta(\overline{X})\,\eta(\overline{Y}),\;\;\mbox{and}\;\;\eta(\overline{X}) =  \overline{g}(\xi,\overline{X}).
\end{equation}
An indefinite almost contact metric manifold $(\overline{M}, \overline{\phi}, \xi, \eta, \overline{g})$ is said to \textit{nearly cosymplectic} if  
\begin{equation}\label{eqz}
         (\overline{\nabla}_{\overline{X}} \overline{\phi})\overline{Y}+(\overline{\nabla}_{\overline{Y}} \overline{\phi})\overline{X}=0, \;\;\forall\, \overline{X}, \overline{Y}\in\Gamma(T\overline{M}),
    \end{equation}
where $\overline{\nabla}$ is the Levi-Civita connection for  $\overline{g}$. Taking  $\overline{Y}=\xi$ in (\ref{eqz}), we get 
\begin{equation}\label{v10}
    \overline{\nabla}_{\overline{X}} \xi =-\overline{H}\,\overline{X},\;\; \forall\,\overline{X}\in\Gamma{(T\overline{M})}.
 \end{equation}
It is easy to see that one can verify the following properties of $\overline{H}$;
 \begin{align}
  & \overline{H}\,\overline{\phi} + \overline{\phi}\,\overline{H}=0,\;\;\overline{H}\xi=0,\;\;\eta\circ \overline{H}=0,\;\;(\overline{\nabla}_{X}\overline{\phi})\xi=\overline{\phi}\,\overline{H}X, \nonumber\\
  \mbox{and}\;\;& \overline{g}(\overline{H}\,\overline{X}, \overline{Y})=-\overline{g}(\overline{X}, \overline{H}\,\overline{Y})\;\;\;\; (\mbox{i.e.}\;\; \overline{H} \;\;\mbox{is skew-symmetric}),
 \end{align}
 for all $\overline{X},\overline{Y}\in \Gamma(T\overline{M})$. Let $\Omega$ denote  the fundamental 2-form of $\overline{M}$ defined by
\begin{equation}
 \Omega(\overline{X}, \overline{Y}) = \overline{g}(\overline{X}, \overline{\phi}\,\overline{Y}),\;\;\overline{X},\;\overline{Y}\in\Gamma(T \overline{M})
\end{equation}
 then,the  1-form $\eta$ and tensor $\overline{H}$ are related as follows;
\begin{lemma}\label{LN}
Let $(\overline{M}, \overline{\phi}, \xi, \eta,\,\overline{g})$ be an indefinite nearly cosymplectic. Then,  
 \begin{equation}\label{RelaOmegaETa}
  d\eta(\overline{X}, \overline{Y}) = \overline{g}( \overline{X}, \overline{H}\,\overline{Y}),\;\;\forall\,\overline{X}, \overline{Y}\in\Gamma(T \overline{M}).
 \end{equation}
  Moreover, $\overline{M}$ is cosymplectic if and only if $\overline{H}$ vanishes identically on $\overline{M}$.
\end{lemma}

Notice that, for all  $\overline{X}$, $\overline{Y}$, $\overline{Z}\in\Gamma(T \overline{M})$,
\begin{equation}
 \overline{g}((\overline{\nabla}_{\overline{Z}}\overline{\phi})\overline{X}, \overline{Y}) = -  \overline{g}( \overline{X}, (\overline{\nabla}_{\overline{Z}}\overline{\phi})\overline{Y}),
\end{equation}
which  means that the tensor $\overline{\nabla}\,\overline{\phi}$ is skew-symmetric. The following lemma is fundamental to the sequel.
\begin{lemma}\label{lems2}
 Let $\overline{M}$ be a nearly cosymplectic manifold, then
 \begin{align}\label{s1}
  (\overline{\nabla}_{\overline{X}}\overline{\phi})\overline{\phi}\, \overline{Y}& =  -\overline{\phi}(\overline{\nabla}_{\overline{X}}\overline{\phi})\overline{Y} - \overline{g}(\overline{Y},\overline{H}\,\overline{X})\xi 
  - \eta(\overline{Y})\overline{H}\,\overline{X}, \\\label{s2}
  (\overline{\nabla}_{\overline{\phi}\,\overline{X}}\overline{\phi})\overline{\phi}\, \overline{Y}& =  -(\overline{\nabla}_{\overline{X}}\overline{\phi})\overline{Y} - \eta(\overline{X})\overline{\phi} \overline{H}\,\overline{Y} + \eta(\overline{Y})\overline{\phi}\overline{H}\,\overline{X},
 \end{align}
for all $\overline{X},\overline{Y}\in \Gamma(T\overline{M})$.
\end{lemma}
\begin{proof}
The proof follows from a straightforward calculation.
\end{proof}

\section{Quasi generalized CR-null submanifolds}\label{AlmostGe}

We recall some basic notions on QGCR-null submanifolds (see \cite{ms} for details).

The structure vector field $\xi$ of an  indefinite almost contact manifold $(\overline{M}, \overline{g})$ can be written according to decomposition (\ref{eq08}) as follows;
         \begin{equation}\label{eq2}
              \xi=\xi_S+\sum_{i=1}^ra_i E_i+\sum_{i=1}^rb_i N_i+\sum_{\alpha=r+1}^nc_\alpha W_\alpha,
         \end{equation}
where $\xi_S$ is a smooth vector field of $S(TM)$ while $a_i=\eta(N_i)$, $b_i=\eta(E_i)$ and $c_\alpha=\epsilon_\alpha\eta(W_\alpha)$ all smooth functions on $\overline{M}$. Here $\epsilon_\alpha=\overline{g}(W_\alpha,W_\alpha)$.

We adopt the definition of quasi generalized CR (QGCR)-lightlike submanifolds given in \cite{ms} for indefinite nearly cosymplectic manifolds.
\begin{definition} \label{def2}{\rm Let $(M,g,S(TM),S(TM^\perp))$  be a null submanifold of an indefinite nearly cosymplectic manifold $(\overline{M}, \overline{g})$. We say that $M$ is quasi generalized CR (QGCR)-null submanifold of $\overline{M}$ if the following conditions are satisfied:
\begin{enumerate}
 \item [(i)] there exist two distributions $D_1$ and $D_2$ of $\textrm{Rad}\,TM$  such that 
        \begin{equation}\label{eq03}
             \mathrm{Rad}\, TM = D_1\oplus D_2, \;\;\overline{\phi} D_1=D_1, \;\;\overline{\phi} D_2\subset S(TM),
        \end{equation}
 \item [(ii)] there exist vector bundles $D_0$ and $\overline{D}$ over $S(TM)$ such that 
         \begin{align} 
              & S(TM)=\{\overline{\phi} D_2 \oplus \overline{D}\}\perp D_0,\\ 
              \mbox{with}\;\;\; &\overline{\phi}D_{0} \subseteq D_{0},\;\;   \overline{D}= \overline{\phi} \, \mathcal{S}\oplus \overline{\phi} \,\mathcal{L}, \label{s81}
         \end{align}    
\end{enumerate}
where $D_0$ is a non-degenerate distribution on $M$, $\mathcal{L}$ and $\mathcal{S}$ are respectively vector subbundles of $l\mathrm{tr}(TM)$  and $S(TM^{\perp})$.
}
\end{definition}

If $D_{1}\neq \{0\}$, $D_0\neq \{0\}$, $D_2\neq \{0\}$ and $\mathcal{S}\neq \{0\}$, then $M$ is called a \textit{proper QGCR-null submanifold}.

A proof of the following Proposition uses similar arguments as in \cite{ms};
\begin{proposition}
 A QGCR-null submanifold $M$ of an indefinite nearly cosymplectic manifold $\overline{M}$ tangent to the structure vector field $\xi$ is a GCR-null submanifold.
 \end{proposition} 
 
Using  (\ref{eq05}), the tangent bundle of any QGCR-null submanifold, $TM$, can be decomposed as 
\begin{equation}
 TM  =D \oplus \widehat{D}, \;\; \mbox{with}\;\;D =D_0\perp D_{1}\;\mbox{and}\;\widehat{D}=\{D_{2}\perp\overline{\phi} D_{2}\} \oplus\overline{D}.\nonumber
\end{equation} 
Unlike for a GCR-null submanifold, in a QGCR-null submanifold,  $D$ is invariant with respect to $\overline{\phi}$ while $\widehat{D}$ is not generally anti-invariant.

Throughout this paper, we suppose that $(M,g,S(TM),S(TM^\perp))$ is a proper QGCR-null submanifold. From the above definition, we can easily deduce the following;
          \begin{enumerate}
               \item condition (i) implies that $\dim(\mathrm{Rad}\, TM)=s\ge  3$,
               \item condition (ii) implies that $\dim(D)\ge  4l\ge 4$ and $\dim(D_2)= \dim(\mathcal{L})$. 
          \end{enumerate}  
 
\begin{definition}[\cite{dh1}]\label{def3}  {\rm
A null submanifold $M$ of a semi-Riemannian manifold $\overline{M}$ is said to be ascreen  if the structure vector field, $\xi$, belongs to $\mathrm{Rad}\, TM \oplus l\mathrm{tr}(T M)$.
}
\end{definition}
From Definition \ref{def3}, Lemma 3.6 and Theorem 3.7 of \cite{ms}, we have  
\begin{theorem}\label{asc}
Let $(M,g,S(TM),S(TM^\perp))$ be an ascreen QGCR-null submanifold of an indefinite nearly cosymplectic manifold $\overline{M}$, then $\xi\in\Gamma(D_{2}\oplus\mathcal{L})$. If $M$
 is a 3-null QGCR submanifold of an indefinite nearly cosymplectic manifold $(\overline{M}, \overline{g})$, then $M$ is ascreen null submanifold if and only if $\overline{\phi}\mathcal{L}=\overline{\phi} D_{2}$.
\end{theorem}
 \begin{proof} 
 The proof follows from straightforward calculation as in \cite{ms}.
 \end{proof}  
 It is crucial to note the following aspects with ascreen QGCR-null submanifold:   item (2) of Definition \ref{def2} implies that $\dim(D)\ge 4l\ge 4$ and $\dim(D_2)=\dim(\mathcal{L})$. Thus $\dim(M)\ge 7$ and $\dim(\overline{M})\ge 11$, and any 7-dimensional ascreen QGCR-null submanifold is 3-null.
 
 In what follows,  we construct an ascreen QGCR-null submanifold of a special  nearly cosymplectic manifold with $\overline{H}=0$ (i.e., $\overline{M}$ is a cosymplectic manifold). Thus, let $(\mathbb{R}_{q}^{2m+1}, \overline{\phi}_{0} ,\xi, \eta, \overline{g})$ denote the manifold $\mathbb{R}_{q}^{2m+1}$ with its usual cosymplectic structure given by 
\begin{align*}
   &\eta = dz,\quad\xi=\partial z,\\
   &\overline{g} =\eta\otimes\eta-\sum_{i=1}^{\frac{q}{2}}(dx^i\otimes dx^i+dy^i\otimes dy^i)+\sum_{i=q+1}^{m}(dx^i\otimes dx^i+dy^i\otimes dy^i),\\
   &\phi_0 (\sum_{i=1}^m(X_i\partial x^i+Y_i\partial y^i)+Z\partial z )=\sum_{i=1}^m(Y_i\partial x^i-X_i\partial y^i),
\end{align*}
where $(x^{i} , y^{i} , z)$ are Cartesian coordinates and $\partial t_{k} = \frac{\partial}{\partial t^{k}}$, for $t\in\mathbb{R}^{2m+1}$.

Now, we use the above structure to construct the following example;
\begin{example}\label{exa1}
{\rm
Let $\overline{M}=(\mathbb{R}_4^{11}, \overline{g})$ be a semi-Euclidean space, with $\overline{g}$ is of signature $(-,-,+,+,+, - ,  -,+,+,+,+)$ with respect to the canonical basis
\begin{equation*}
 (\partial x_{1},\partial x_{2},\partial x_3,\partial x_4,\partial x_5,\partial y_1,\partial y_2,\partial y_3,\partial y_4,\partial y_{5},\partial z).
\end{equation*}
Let $(M,g)$ be a submanifold of $\overline{M}$ given by 
\begin{equation*}
 x^1=y^4,\;\; y^1=-x^4,\;\; z= x^2\sin\theta + y^2\cos\theta \;\;\mbox{and}\;\; y^5=(x^5)^{\frac{1}{2}},
 \end{equation*}
 where $\theta\in(0,\frac{\pi}{2})$. By direct calculations, we can see that the vector fields
\begin{align*} 
 E_1 & =\partial x_4+\partial y_1,\;\;\; E_2=\partial x_1-\partial y_4,\\
 E_3 & =\sin\theta \partial x_2 +\cos\theta \partial y_2+\partial z, \; \;\; X_1=2y^5\partial x_5+\partial y_5,\\ 
 X_2 & =-\cos\theta \partial x_2 +\sin\theta \partial y_2,\;\;\;  X_3=\partial y_3,\; X_4=\partial x_3, 
\end{align*}
form a local frame of $TM$. Then $\mathrm{Rad} \, TM$ is spanned by $\{E_1, E_2, E_3\}$, and therefore, $M$ is 3-null. Further, $\overline{\phi}_0 E_1=E_2$, therefore we set $D_1=\mbox{Span}\{E_1,E_2\}$. Also $\overline{\phi}_0 E_3=-X_2$ and thus $D_2=\mathrm{Span}\{E_3\}$. It is easy to see that $\overline{\phi}_0 X_3=X_4$, so we set $D_0=\mathrm{Span}\{X_3,X_4\}$. On the other hand, following direct calculations, we have 
\begin{align} 
 N_1 & =\frac{1}{2}(\partial x_4-\partial y_1),\;\; \;  N_2=\frac{1}{2}(-\partial x_1-\partial y_4),\nonumber\\
 N_3 & =\frac{1}{2}(-\sin\theta \partial x_2 -\cos\theta \partial y_2+\partial z),\;\; \;  W=\partial x_5-2y^5\partial y_5, \nonumber
\end{align}
from which $l\mathrm{tr}(TM)=\mathrm{Span}\{N_1,N_2,N_3\}$ and 
$S(TM^\perp)=\mathrm{Span}\{W\}$. Clearly, $\overline{\phi}_0 N_2=-N_1$. 
Further, $\overline{\phi}_0 N_3=\frac{1}{2} X_2$ and thus 
$\mathcal{L}=\mbox{Span}\{N_3\}$. Notice that 
$\overline{\phi}_0N_3=-\frac{1}{2}\overline{\phi}_0 E_3$ and therefore 
$\overline{\phi}_0\mathcal{L}=\overline{\phi}_0 D_2$. Also, $\overline{\phi}_0 
W=-X_1$ and therefore $\mathcal{S}=\mbox{Span}\{W\}$. Finally, we calculate 
$\xi$ as follows; Using Theorem \ref{asc} we have $\xi=a E_3+b N_3$. Applying 
$\overline{\phi}_0$ to this equation we obtain $a\overline{\phi}_0 
E_3+b\overline{\phi}_0 N_3=0$. Now, substituting for $\overline{\phi}_0 E_3$ and 
 $\overline{\phi}_0 N_3$ in this equation we get $2a=b$, from which we get 
$\xi=\frac{1}{2}(E_3+2N_3)$. Since $\overline{\phi}_0\xi=0$ and 
$\overline{g}(\xi,\xi)=1$, we conclude that $(M,g)$ is an ascreen QGCR-null 
submanifold of $\overline{M}$.}
\end{example}
\begin{proposition}
There exist no co-isotropic, isotropic or totally null proper QGCR-null submanifolds of an indefinite nearly cosymplectic manifold.
\end{proposition}

\section{Umbilical and Geodesic ascreen QGCR-null submanifolds}\label{existence}

In this section, we prove two main theorems concerning totally umbilical, totally geodesic and irrotational ascreen QGCR-null submanifolds of  $\overline{M}$. An indefinite nearly cosymplectic manifold $\overline{M}$ is called an \textit{indefinite nearly cosymplectic space form}, denoted by $\overline{M}(\overline{c})$, if it has the constant $\overline{\phi}$-sectional curvature $\overline{c}$. The curvature tensor $\overline{R}$ of the indefinite nearly cosymplectic space form $\overline{M}(\overline{c})$ is given by \cite{endo}:
 \begin{align}\label{s9}
 4\overline{R}(\overline{X},\overline{W},\overline{Z},\overline{Y})& = \overline{g}((\overline{\nabla}_{\overline{W}}\overline{\phi})\overline{Z},(\overline{\nabla}_{\overline{X}}\overline{\phi})\overline{Y})-\overline{g}((\overline{\nabla}_{\overline{W}}\overline{\phi})\overline{Y},(\overline{\nabla}_{\overline{X}}\overline{\phi})\overline{Z})\nonumber\\
                        & -2\overline{g}((\overline{\nabla}_{\overline{W}}\overline{\phi})\overline{X},(\overline{\nabla}_{\overline{Y}}\overline{\phi})\overline{Z})+\overline{g}(\overline{H}\,\overline{W},\overline{Z})\overline{g}(\overline{H}\,\overline{X},\overline{Y})\nonumber\\
                        & -\overline{g}(\overline{H}\,\overline{W},\overline{Y})\overline{g}(\overline{H}\,\overline{X},\overline{Z})-2\overline{g}(\overline{H}\,\overline{W},\overline{X})\overline{g}(\overline{H}\,\overline{Y},\overline{Z})\nonumber\\
                        & -\eta(\overline{W})\eta(\overline{Y})\overline{g}(\overline{H}\,\overline{X},\overline{H}\,\overline{Z})+\eta(\overline{W})\eta(\overline{Z})\overline{g}(\overline{H}\,\overline{X},\overline{H}\,\overline{Y})\nonumber\\
                        & +\eta(\overline{X})\eta(\overline{Y})\overline{g}(\overline{H}\,\overline{W},\overline{H}\,\overline{Z})-\eta(\overline{X})\eta(\overline{Z})\overline{g}(\overline{H}\,\overline{W},\overline{H}\,\overline{Y})\nonumber\\
                        &+\overline{c}\{\overline{g}(\overline{X},\overline{Y})\overline{g}(\overline{Z},\overline{W})-\overline{g}(\overline{Z},\overline{X})\overline{g}(\overline{Y},\overline{W})\nonumber\\
                        &+\eta(\overline{Z})\eta(\overline{X})\overline{g}(\overline{Y},\overline{W})-\eta(\overline{Y})\eta(\overline{X})\overline{g}(\overline{Z},\overline{W})\nonumber\\
                        &+\eta(\overline{Y})\eta(\overline{W})\overline{g}(\overline{Z},\overline{X})-\eta(\overline{Z})\eta(\overline{W})\overline{g}(\overline{Y},\overline{X})\nonumber\\
                        &+\overline{g}(\overline{\phi}\,\overline{Y},\overline{X})\overline{g}(\overline{\phi}\,\overline{Z},\overline{W})-\overline{g}(\overline{\phi}\,\overline{Z},\overline{X})\overline{g}(\overline{\phi}\, \overline{Y},\overline{W})\nonumber\\
                        &-2\overline{g}(\overline{\phi}\,\overline{Z},\overline{Y})\overline{g}(\overline{\phi}\,\overline{X},\overline{W})\},
 \end{align}
for all $\overline{X},\overline{Y},\overline{Z},\overline{W}\in\Gamma(T\overline{M})$.

  Notice that $D_0$ and $\overline{\phi}\mathcal{S}$ are orthogonal and non-degenerate subbundles of $TM$ and that when $M$ is ascreen QGCR-null submanifold, we observe that 
\begin{equation}\label{M11}
 \eta(X)=\eta(Z)=0,\quad \forall\,X\in\Gamma(D_0),\quad Z\in\Gamma(\overline{\phi}\mathcal{S}).
\end{equation}

\begin{theorem}\label{cos}
 Let $(M,g,S(TM)$, $S(TM^\perp))$ be a totally umbilical or totally geodesic ascreen QGCR-null submanifold of an indefinite nearly cosymplectic space form  $\overline{M}(\overline{c})$, of pointwise constant $\overline{\phi}$-sectional curvature $\overline{c}$, such that $D_{0}$ and $\overline{\phi}\mathcal{S}$ are space-like and parallel distributions with respect to $\nabla$. Then,  $\overline{c}\ge 0$.  Equality occurs when $\overline{M}(\overline{c})$ is an indefinite cosymplectic space form.
\end{theorem}

 \begin{proof}
Let $X$ and $Z$ be vector fields in $D_{0}$ and $\overline{\phi}\mathcal{S}$, respectively. Replacing $\overline{W}$ with $\overline{\phi} X$ and $\overline{Y}$ with $\overline{\phi}Z$ in (\ref{s9}), we get
  \begin{align}\label{s10}
  4 \overline{R}(X,&\overline{\phi}X,Z,\overline{\phi} Z) = \overline{g}((\overline{\nabla}_{\overline{\phi} X}\overline{\phi})Z,(\overline{\nabla}_X\overline{\phi})\overline{\phi} Z)\nonumber\\
                  &-\overline{g}((\overline{\nabla}_{ \overline{\phi}X}\overline{\phi})\overline{\phi}Z,(\overline{\nabla}_X\overline{\phi})Z)-2\overline{g}((\overline{\nabla}_{ \overline{\phi}X}\overline{\phi})X,(\overline{\nabla}_{ \overline{\phi}Z}\overline{\phi})Z)\nonumber\\
                  &+\overline{g}(\overline{H}\,\overline{\phi} X,Z)\overline{g}(\overline{H}X,\overline{\phi}Z)-\overline{g}(\overline{H}\,\overline{\phi} X, \overline{\phi}Z)\overline{g}(\overline{H}X,Z)\nonumber\\
                  &-2\overline{g}(\overline{H}\,\overline{\phi}X,X)\overline{g}(\overline{H}\,\overline{\phi} Z,Z)-2\overline{c}g( \overline{\phi}Z,\overline{\phi}Z)g(\overline{\phi} X, \overline{\phi}X).
  \end{align}
  Considering the first three terms on the right hand side of (\ref{s10}), we have  
   \begin{align}\label{s11}
    \overline{g}((\overline{\nabla}_{\overline{\phi} X}\overline{\phi})Z,(\overline{\nabla}_X\overline{\phi})\overline{\phi} Z)=& -\overline{g}((\overline{\nabla}_{Z}\overline{\phi})\overline{\phi} X,(\overline{\nabla}_X\overline{\phi}) \overline{\phi}Z).
   \end{align}
  Applying (\ref{s1}) of Lemma \ref{lems2} on (\ref{s11}) we derive
   \begin{align}\label{s14}
   &\overline{g}((\overline{\nabla}_{ \overline{\phi}X}\overline{\phi})Z,(\overline{\nabla}_X\overline{\phi})\overline{\phi} Z)= -\overline{g}((\overline{\nabla}_{Z}\overline{\phi})\overline{\phi}X,(\overline{\nabla}_X\overline{\phi})\overline{\phi} Z)\nonumber\\
                                                 &=\overline{g}((\overline{\nabla}_X\overline{\phi})Z,(\overline{\nabla}_X\overline{\phi})Z)-\overline{g}(\overline{\phi}Z,\overline{H}X)^2+\overline{g} (Z,\overline{H}X)^2.
\end{align}  
In a similar way, using (\ref{s2}) of Lemma \ref{lems2}, we get
\begin{equation}\label{s12}
 -\overline{g}((\overline{\nabla}_{ \overline{\phi}X}\overline{\phi})\overline{\phi}Z,(\overline{\nabla}_X\overline{\phi})Z)=\overline{g}((\overline{\nabla}_X\overline{\phi})Z,(\overline{\nabla}_X\overline{\phi})Z),
\end{equation}
and 
\begin{equation}\label{s13}
 -2\overline{g}((\overline{\nabla}_{ \overline{\phi}X}\overline{\phi})X,(\overline{\nabla}_{\overline{\phi} Z}\overline{\phi})Z)=0
\end{equation}
Now substituting (\ref{s14}), (\ref{s12}) and (\ref{s13}) in (\ref{s10}), we get
  \begin{align*}
    &4 \overline{R}(X,\overline{\phi}X,Z,\overline{\phi} Z) = 2\overline{g}((\overline{\nabla}_X\overline{\phi})Z,(\overline{\nabla}_X\overline{\phi})Z)-\overline{g} (\overline{\phi}Z,\overline{H}X)^2\nonumber\\
    &+\overline{g} (Z,\overline{H}X)^2+\overline{g}(\overline{H}\,\overline{\phi} X,Z)\overline{g}(\overline{H}X,\overline{\phi}Z)-\overline{g}(\overline{H}\,\overline{\phi} X, \overline{\phi}Z)\overline{g}(\overline{H}X,Z)\nonumber\\
                  &-2\overline{g}(\overline{H}\,\overline{\phi}X,X)\overline{g}(\overline{H}\,\overline{\phi} Z,Z)-2\overline{c}g( \overline{\phi}Z,\overline{\phi}Z)g(\overline{\phi} X, \overline{\phi}X),
  \end{align*}
 from which we obtain
 \begin{align}\label{s74}
   2\overline{R}(X,\overline{\phi} X,Z, \overline{\phi}Z) &=\overline{g}((\overline{\nabla}_X\overline{\phi})Z,(\overline{\nabla}_X\overline{\phi})Z)+\overline{g} (Z,\overline{H}X)^2\nonumber\\
   &-\overline{c}g(Z,Z)g(X,X).
 \end{align} 
Then using the facts $D_{0}$ and $\overline{\phi}\mathcal{S}$ are space-like and parallel with respect to $\nabla$, we have
$$
(\overline{\nabla}_{Z}\overline{\phi})X = (\nabla_{Z}\overline{\phi})X\in\Gamma(D_{0}),
$$
and (\ref{s74}) reduces to 
\begin{equation}\label{M30}
  2\overline{R}(X,\overline{\phi} X,Z,\overline{\phi}Z)=||(\nabla_{Z}\overline{\phi})X||^{2}+\overline{g}(Z,\overline{H}X)^{2}-\overline{c}||X||^{2}||Z||^{2},
\end{equation}
where $||.||$ denotes the norm on $D_{0}\perp \overline{\phi}\mathcal{S}$ with respect to $g$.

On the other hand, if we set  $\overline{W}=\overline{\phi} X$ and $\overline{Y}=\overline{\phi} Z$ in (\ref{s8}), we have
\begin{align}\label{s17}
  &\overline{R}(X,\overline{\phi}X,Z,\overline{\phi} Z)\nonumber\\
  &=\;\overline{g}((\nabla_X h^s)(\overline{\phi} X,Z),\overline{\phi} Z)-\overline{g}((\nabla_{ \overline{\phi}X} h^s)(X,Z),\overline{\phi} Z),
\end{align}
where, 
\begin{equation}\label{s18}
 (\nabla_X h^s)(\overline{\phi} X,Z)=\nabla_X^sh^s(\overline{\phi} X,Z)-h^s(\nabla_X \overline{\phi}X,Z)-h^s( \overline{\phi}X,\nabla_X Z).
\end{equation}
By the fact that $M$ is totally umbilical in $\overline{M}$, we have $h^s(\overline{\phi} X,Z)=0$. Thus using (\ref{s191}), equation (\ref{s18}) becomes
\begin{align}\label{s19}
 (\nabla_X h^s)(\overline{\phi} X,Z)&=-h^s(\nabla_X\overline{\phi} X,Z)-h^s(\overline{\phi} X,\nabla_X Z)\nonumber\\
                         =&-g(\nabla_X\overline{\phi} X,Z)\mathscr{H}^s-g(\overline{\phi} X,\nabla_X Z)\mathscr{H}^s.
\end{align}
Differentiating $\overline{g}(\overline{\phi} X,Z)=0$ covariantly with respect to $X$ and then applying (\ref{eq11}), we obtain 
\begin{equation}\label{s21}
 g(\nabla_X \overline{\phi}X,Z)+g(\overline{\phi} X,\nabla_X Z)=0.
\end{equation}
Substituting (\ref{s21}) in (\ref{s19}), gives
\begin{equation}\label{s22}
 (\nabla_X h^s)(\overline{\phi} X,Z)=0.
\end{equation}
Similarly, 
\begin{equation}\label{s23}
 (\nabla_{ \overline{\phi}X} h^s)(X,Z)=0.
\end{equation}
Then, substituting (\ref{s22}) and (\ref{s23}) in (\ref{s17}), we get
\begin{equation}\label{s25}
 \overline{R}(X,\overline{\phi} X,Z,\overline{\phi} Z)=0.
\end{equation}
Substituting (\ref{s25}) in (\ref{M30}), gives
 \begin{equation}\label{s30}
     \overline{c}||X||^{2}||Z||^{2}=||(\nabla_{Z}\overline{\phi})X||^{2}+\overline{g}(Z,\overline{H}X)^{2}\ge 0,
 \end{equation} 
which implies that $\overline{c}\ge 0$. When the ambient manifold is cosymplectic, then $\nabla\overline{\phi}=0$ and $d\eta=0$ \cite{bl2} and in this case $\overline{c}=0$.
\end{proof}
\begin{example}\label{exa2}
\rm{
Let $M$  be an ascreen QGCR-null submanifold in Example \ref{exa1}  Applying (\ref{eq11}) and Koszul's formula (see \cite{db}) to Example \ref{exa1} we obtain
 \begin{align}\label{M1}
  &h_{i}^{l}(X,Y)=0\quad\forall\, X,Y\in\Gamma(TM),\quad \mbox{where}\quad i=1,2,3,\nonumber\\
  & \epsilon_{4}h_{4}^{s}(X_{1},X_{1})=2\quad \mbox{and}\quad h_{4}^{s}(X,Y)=0,\quad\forall\,X\neq X_{1},Y\neq X_{1}.
 \end{align}
Using (\ref{h1}), (\ref{M1}) and $\epsilon_{4}=\overline{g}(W,W)=1+4(y^{5})^2$, we also derive 
\begin{equation}\label{M2}
 h(X_{1},X_{1})=\frac{2}{1+4(y^{5})^{2}}W.
\end{equation}
We remark that $M$ is not totally geodesic.  From (\ref{M2}) and (\ref{eq17}) we note that $M$ is totally umbilical with 
\begin{equation*}
 \mathcal{H}=\frac{2}{(1+4(y^{5})^{2})^{2}}W.
\end{equation*}
By straightforward calculations we also have 
\begin{equation*}
 \nabla_{X_{1}}X_{1}=4y^{5}X_{1}\quad\mbox{and} \quad  \nabla_{X_{i}}X_{j}=0\quad \forall\,i,j\neq1.
\end{equation*}
Thus, $D_{0}$ and $\overline{\phi}\mathcal{S}$ are parallel distributions with respect to $\nabla$. Hence, $M$ satisfies Theorem \ref{cos} and  $\overline{c}=0$.}
\end{example} 
\begin{corollary}
  Let $(M,g,S(TM)$, $S(TM^\perp))$ be a totally umbilical or totally geodesic ascreen QGCR-null submanifold of an indefinite  cosymplectic space form  $\overline{M}(\overline{c})$ of pointwise constant $\overline{\phi}$-sectional curvature $\overline{c}$. Then, $\overline{c}=0$.
\end{corollary}
A null submanifold $M$ of a semi-Riemannian manifold $(\overline{M},\overline{g})$ is called irrotational \cite{ds2} if $\overline{\nabla}_{X}E\in\Gamma(TM)$, for any $E\in\Gamma(\mathrm{Rad} \, TM))$ and $X\in\Gamma(TM)$. Equivalently, $M$ is irrotational if 
\begin{equation}\label{ms40}
 h^{l}(X,E)= h^{s}(X,E)=0,
\end{equation}
for all $X\in\Gamma(TM)$ and $ E\in\Gamma(\mathrm{Rad} \,TM)$.
\begin{theorem}
 Let $(M,g,S(TM)$, $S(TM^\perp))$ be an irrotational ascreen QGCR-null submanifold of an indefinite nearly cosymplectic space form  $\overline{M}(c)$ of pointwise constant $\overline{\phi}$-sectional curvature $\overline{c}$. Then, $\overline{c}\le 0$ or $\overline{c}\ge 0$.  Equality holds when $\overline{M}(\overline{c})$ is an indefinite cosymplectic space form.
\end{theorem}
\begin{proof}
 By setting  $\overline{Y}=\overline{Z}=E$, $\overline{X}$ and $\overline{W}=\overline{\phi}E$ in (\ref{s8}), we get 
       \begin{align}\label{ms41}
              &  \overline{R}(X,\overline{\phi}E,E,E) =\overline{g}((\nabla_X h^l)(\overline{\phi}E,E),E)-\overline{g}((\nabla_{\overline{\phi}E} h^l)(X,E),E)\nonumber\\
                                     &+\overline{g}((\nabla_X h^s)(\overline{\phi}E,E),E)-\overline{g}((\nabla_{\overline{\phi}E} h^s)(X,E),E)
      \end{align}
for any $X\in \Gamma(TM)$ and $E\in\Gamma(\mathrm{Rad} \,TM)$. Then, using the fact that $M$ is irrotational, (\ref{ms41}) reduces to 
\begin{equation}\label{ms42}
  \overline{R}(X,\overline{\phi}E,E,E)=0,\;\;\forall\;X\in\Gamma(TM).
\end{equation}
On the other hand, setting $\overline{Y}=\overline{W}=E$ and $\overline{Z}=\overline{\phi}E$ in (\ref{s9}) and simplifying ,  we get 
  
 \begin{align}\label{ms43}
  &\overline{R}(X,E,\overline{\phi} E,E) =-3\overline{g}((\overline{\nabla}_{E}\overline{\phi})\overline{\phi} E,(\overline{\nabla}_{E}\overline{\phi})X)\nonumber\\
                   &-\eta(E)^{2}\overline{g}(\overline{H}X,\overline{H}\,\overline{\phi}E)+4\overline{c}\eta(E)^{2}\overline{g}(X,\overline{\phi} E).
 \end{align}
Now, using (\ref{ms42}) and (\ref{ms43}), we get
 \begin{align}\label{ms44}
 &4\overline{c}\eta(E)^{2}\overline{g}(X,\overline{\phi} E)\nonumber\\
 &=3\overline{g}((\overline{\nabla}_{E}\overline{\phi})\overline{\phi} E,(\overline{\nabla}_{E}\overline{\phi})X)+\eta(E)^{2}\overline{g}(\overline{H}X,\overline{H}\,\overline{\phi}E).
 \end{align}
Replacing $X$ with $\overline{\phi}E$ in (\ref{ms44}) and the using (\ref{s1}) of Lemma \ref{lems2} to the resulting equation gives
\begin{equation}\label{ms45}
 \overline{c}\eta(E)^{2}\overline{g}(\overline{\phi} E,\overline{\phi} E)=\eta(E)^{2}\overline{g}(\overline{H}\,\overline{\phi} E,\overline{H}\,\overline{\phi}E).
\end{equation}
Since $M$ is ascreen QGCR-null submanifold, there exists $E\in\Gamma(D_{2})$ such that $\eta(E)=b\neq0$, and thus (\ref{ms45}) simplifies to
\begin{equation}\label{ms455}
 \overline{c}=-\frac{1}{b^{2}}\overline{g}(\overline{H} E,\overline{H}E)=\frac{1}{b^{2}}d\eta(E,\overline{H} E).
\end{equation}
We observe that $\overline{c}=0$ if either $d\eta=0$ (i.e., $\overline{M}(c)$ is cosymplectic space form \cite{bl2}) or $\overline{H} E$ is a null vector field. The second case implies that $\overline{H} E$ belongs to $\mathrm{Rad} \,TM$ or $l\mathrm{tr}(TM)$. If $\overline{H} E\in\Gamma(\mathrm{Rad} \,TM)$, then there exists a non zero smooth function $\kappa$ such that $\overline{H} E=\kappa E$, for some arbitrary $E\in\Gamma(\mathrm{Rad} \,TM)$. Taking the $\overline{g}$-product of $\overline{H}E=\kappa E$ with $\xi$ leads to $0=\kappa\eta(E)$, from which $\eta(E)=0$. Since $M$ is ascreen QGCR-null submanifold, then, there is $E\in\Gamma(D_{2})$ such that $\eta(E)\neq0$, hence a contradiction. Similar reasoning can be applied if $\overline{H} E\in\Gamma(l\mathrm{tr}(TM))$. Therefore, $\overline{c}=0$ only if $\overline{H}E=0$ (i.e., $d\eta=0$) which occurs when $\overline{M}(c)$ is cosymplectic space form \cite{bl2}. It turns out that $\overline{c}\le 0$ or $\overline{c}\ge 0$ depending on whether $\overline{H}E$ is space-like or time-like vector field respectively.
\end{proof} 
\begin{corollary}
 Let $(M,g,S(TM)$, $S(TM^\perp))$ be an irrotational ascreen QGCR-null submanifold of an indefinite cosymplectic space form  $\overline{M}(\overline{c})$ of pointwise constant $\overline{\phi}$-sectional curvature $\overline{c}$. Then, $\overline{c}=0$.
\end{corollary}
It is easy to see from (\ref{M2}) that $h^{l}(X,E)= h^{s}(X,E)=0$ and hence $M$ given in Example \ref{exa2} is an  irrotational ascreen QGCR-null submanifold of an indefinite cosymplectic space form $\overline{M}(\overline{c})$. As is proved in that example $\overline{c}=0$.

\section{Mixed totally geodesic QGCR-null submanifolds}\label{mixed}

\begin{definition}{\rm
 A QGCR-null submanifold of an indefinite nearly cosymplectic manifold $(\overline{M}, \overline{g})$ is called \textit{mixed totally geodesic} QGCR-null submanifold if its second fundamental form, $h$, satisfies $h(X,Y)=0$, for any $X\in\Gamma(D)$ and $Y\in\Gamma(\widehat{D})$.
 }
\end{definition}
We will need the following lemma in the next theorem.
\begin{lemma}\label{lems7}
 Let $(M,g,S(TM),S(TM^\perp))$ be any 3-null proper ascreen QG CR-null submanifold of an indefinite nearly cosymplectic manifold $(\overline{M}, \overline{g})$. Then, 
 $$2\eta(E)\eta(N)=1,
 $$
for any $E\in\Gamma(D_{2})$ and $N\in\Gamma(\mathcal{L})$.
\end{lemma}
\begin{proof}
The proof follows from straightforward calculations using $\overline{g}(\xi,\xi)=1$ and $\xi=\eta(N)E+\eta(E)N$.
\end{proof}
\begin{theorem}
 Let $(M,g,S(TM),S(TM^\perp))$ be a 3-null proper ascreen QG CR-null submanifold of an indefinite nearly cosymplectic manifold $(\overline{M}, \overline{g})$. Then, $M$ is mixed totally geodesic if and only if $h_{\alpha}^{s}(X,Y)=0$ and $A_{E_{i}}^* X=0$, for all $X\in\Gamma(D)$,  $Y\in\Gamma(\widehat{D})$, $W_{\alpha}\in\Gamma(S(TM^{\perp}))$ and $E_i\in\Gamma(\mathrm{Rad} \, TM)$.
\end{theorem}
\begin{proof}
 By the defintion of ascreen QGCR-null submanifold, $M$ is mixed geodesic if  
        \begin{equation}\label{t1}
        \overline{g}(h(X,Y),W_{\alpha})=\overline{g}(h(X,Y),E_{i})=0,
        \end{equation}
 for all $X\in\Gamma(D)$, $Y\in\Gamma(\widehat{D})$, $W_{\alpha}\in\Gamma(S(TM^{\perp}))$ and $E_i\in\Gamma(\mathrm{Rad} \, TM)$. Now, by virtue of  (\ref{h1}) and the first equation of (\ref{t1}), we have 
 \begin{equation*}
  0=\overline{g}(h(X,Y),W_{\alpha})=\epsilon_{\alpha}h_{\alpha}^{s}(X,Y),
 \end{equation*}
 from which $h_{\alpha}^{s}(X,Y)=0$, since $\epsilon_{\alpha}\neq0$. On the other hand, using the second equation of (\ref{t1}),  (\ref{eq11}) and (\ref{eq50}) we derive
 \begin{equation}\label{ms23}
  \overline{g}(h(X,Y),E_{i})=\overline{g}(\overline{\nabla}_{X} Y,E_{i})=-\overline{g}(Y,\overline{\nabla}_{X} E_{i})=g(Y,A_{E_{i}}^* X)=0.
 \end{equation}
  Since $D=D_{0}\perp D_{1}$ and $\widehat{D}=\{D_{2}\perp\overline{\phi} D_{2}\} \oplus\overline{D}$, we observe that $A_{E_{i}}^* X\notin\Gamma(\overline{\phi} D_{2})$ or $\overline{\phi} \mathcal{L}$. In fact, let suppose that $A_{E_{i}}^* X\notin\Gamma(\overline{\phi} D_{2})$, then there exists a non-vanishing smooth function $\ell$ such that $A_{E_{i}}^* X=\ell \overline{\phi} E$, for $E\in\Gamma(D_{2})$. Thus, 
 \begin{equation}\label{ms24}
  0=g(Y,A_{E_{i}}^* X)=\ell g(Y,\overline{\phi} E),\;\;\forall\;Y\in\Gamma(\widehat{D}).
 \end{equation}
Taking $Y=\overline{\phi}N$ in (\ref{ms24}), where $N\in\Gamma(\mathcal{L})$ and using Lemma \ref{lems7}, we have
      \begin{equation}\nonumber
       0=g(Y,A_{E_{i}}^* X)=\ell g(\overline{\phi}N,\overline{\phi} E)=\ell(1-\eta(E)\eta(N))=\frac{1}{2}\ell,
      \end{equation}
      
which is a contradiction, since $\ell\neq0$. Hence $A_{E_{i}}^* X\notin\Gamma(\overline{\phi} D_{2}\oplus\overline{\phi} \mathcal{L})$. Moreover, $A_{E_{i}}^* X\notin\Gamma(\overline{\phi} \mathcal{S})$ since if $A_{E_{i}}^* X\in\Gamma(\overline{\phi} \mathcal{S})$, then there is a non-vanishing smooth function $\omega$ such that $A_{E_{i}}^* X=\omega\overline{\phi}W_{\alpha}$. Taking the $\overline{g}$-product of this equation with respect to $Y=\overline{\phi}W_{\alpha}$ and using the fact that $\eta(W_{\alpha})=0$, we get 
      \begin{equation}\nonumber
      0=g(Y,A_{E_{i}}^* X)=\omega\overline{g}(\overline{\phi}W_{\alpha},\overline{\phi}W_{\alpha})=\omega\overline{g}  (W_{\alpha},W_{\alpha})=\omega\epsilon_{\alpha},
     \end{equation}
which is a contradiction, since $\epsilon_{\alpha}\neq0$ and $\omega\neq0$. Hence, $A_{E_{i}}^* X\notin\Gamma(\{\overline{\phi} D_{2}\oplus\overline{\phi} \mathcal{L}\}\perp \overline{\phi}\mathcal{S})$, which implies that $A_{E_{i}}^* X\in\Gamma(D_{0})$. Since $A_{E_{i}}^* X\in\Gamma(D_{0})$, then the non-degeneracy of $D_{0}$ implies that there exists some $Z\in\Gamma(D_{0})$ such that $g(A_{E_{i}}^* X, Z)\neq 0$. But using (\ref{eq50})and (\ref{eq11}), together with the fact that $M$ is mixed geodesic we derive
\begin{equation}\label{ms25}
 g(A_{E_{i}}^* X, Z)=-g(\nabla_{X}E_{i},Z)=\overline{g}(E_{i},\overline{\nabla}_{X}Z)=\overline{g}(E_{i},\nabla_{X}Z)=0,
\end{equation}
which is a contradiction. Thus $A_{E_{i}}^* X\notin\Gamma(\{\overline{\phi} D_{2}\oplus\overline{\phi} \mathcal{L}\}\perp \overline{\phi}\mathcal{S}\perp D_{0})$, i.e., $A_{E_{i}}^* X=0$. The converse is obvious.
\end{proof}

\begin{corollary}
 Let $(M,g,S(TM),S(TM^\perp))$ be a proper ascreen QGCR-null submanifold of an indefinite nearly cosymplectic manifold $(\overline{M}, \overline{g})$. Then, if $M$ is mixed totally geodesic then $h_{i}^{l}(X,E_{i})=0$ and $\varphi_{\alpha i}(X)=0$, for all $X\in\Gamma(D)$ and $E_{i}\in\Gamma(D_{2})$.
\end{corollary}
\begin{definition}{\rm
 A QGCR-null submanifold of an indefinite nearly cosymplectic manifold $(\overline{M}, \overline{g})$ is called $D$-totally geodesic QGCR-null submanifold if its second fundamental form $h$ satisfies
 $$
 h(X,Y)=0,\;\;\forall X,\,Y\in\Gamma(D).
 $$
 }
\end{definition} 
 Since $M$ is ascreen QGCR-null submanifold, we have $\overline{g}(X,\xi)=0$, for all $X\in\Gamma(D)$. Applying $\overline{\nabla}_{Y}$ to $\overline{g}(X,\xi)=0$ we get
 \begin{equation}\label{ms30}
  \eta(\overline{\nabla}_{Y}X)=-\overline{g}(X,\overline{\nabla}_{Y}\xi)=\overline{g}(X,\overline{H}Y).
 \end{equation}
Interchanging $X$ and $Y$ in (\ref{ms30}), and then adding the resulting equation to (\ref{ms30}), gives
 \begin{equation}\label{lems11}
  \eta(\overline{\nabla}_{X}Y)+\eta(\overline{\nabla}_{Y}X)=\overline{g}(Y,\overline{H}X)+\overline{g}(X,\overline{H}Y)=0.
 \end{equation}  
\begin{theorem}
 Let $(M,g,S(TM),S(TM^\perp))$ be a proper ascreen QGCR-null submanifold of an indefinite nearly cosymplectic manifold $(\overline{M}, \overline{g})$. Then, $M$ is $D$-totally geodesic if and only if $\overline{\phi}h^{l}(X,\overline{\phi}E)$ and $\overline{\phi}h^{s}(X,\overline{\phi}W)$ respectively have no components along $l\mathrm{tr}(TM)$ and $S(TM^\perp)$, while both $\nabla_{X} \overline{\phi} E$ and $\nabla_{X} \overline{\phi} W\notin\Gamma(D_{0})$ for all  $X\in\Gamma(D)$, $E\in\Gamma(\mathrm{Rad} \, TM)$ and $W\in\Gamma(\mathcal{S})$.
\end{theorem}
\begin{proof}
 By the definition of ascreen QGCR-null submanifold, $M$ is $D$ geodesic if and only if $\overline{g}(h(X,Y),E)=\overline{g}(h(X,Y),W)=0$, for all $X,Y\in\Gamma(D)$, $W_{\alpha}\in\Gamma(S(TM^{\perp}))$ and $E\in\Gamma(\mathrm{Rad} \, TM)$.

 Using (\ref{eq11}) and (\ref{equa2}), we derive
 \begin{align*}
  \overline{g}(h(X,Y),E)   &=\overline{g}(\overline{\nabla}_{X}Y,E)=\overline{g}(\overline{\phi}\,\overline{\nabla}_{X}Y,\overline{\phi}E)-\overline{g}(Y,\overline{\nabla}_{X}\xi)\overline{g}(E,\xi), 
 \end{align*}
from which when we apply (\ref{v10}) we get
\begin{equation}\label{ms31}
 \overline{g}(h(X,Y),E)=\overline{g}(\overline{\phi}\,\overline{\nabla}_{X}Y,\overline{\phi}E)+\overline{g}(Y,\overline{H}X)\overline{g}(E,\xi).
\end{equation}
Interchanging $X$ and $Y$ in (\ref{ms31}) and considering the fact that $h$ is symmetric we get 
\begin{equation}\label{ms32}
 \overline{g}(h(X,Y),E)=\overline{g}(\overline{\phi}\,\overline{\nabla}_{Y}X,\overline{\phi}E)+\overline{g}(X,\overline{H}Y)\overline{g}(E,\xi).
\end{equation}
Summing (\ref{ms31}) and  (\ref{ms32}), and then applying (\ref{lems11}), we have
\begin{equation}\label{ms33}
 2\overline{g}(h(X,Y),E)=\overline{g}(\overline{\phi}\,\overline{\nabla}_{X}Y,\overline{\phi}E)+\overline{g}(\overline{\phi}\,\overline{\nabla}_{Y}X,\overline{\phi}E).
\end{equation}
Now, applying the nearly cosymplectic condition in (\ref{eqz}) to (\ref{ms33}), leads to 
\begin{equation}\label{ms34}
 2\overline{g}(h(X,Y),E)=\overline{g}(\overline{\nabla}_{X}\overline{\phi}Y,\overline{\phi}E)+\overline{g}(\overline{\nabla}_{Y}\overline{\phi}X,\overline{\phi}E).
\end{equation}
From (\ref{ms34}) and (\ref{eq11}) we derive
\begin{align}\label{ms35}
 2\overline{g}(h(X,Y),E)&=\overline{g}(\overline{\nabla}_{X}\overline{\phi}Y,\overline{\phi}E)+\overline{g}(\overline{\nabla}_{Y}\overline{\phi}X,\overline{\phi}E)\nonumber\\ 
  &-\overline{g}(\overline{\phi}Y,h(X,\overline{\phi}E))-\overline{g}(\overline{\phi}X,h(Y,\overline{\phi}E))
\end{align}
If we let $X,Y\in\Gamma(D_{1})$ in (\ref{ms35}), we obtain
\begin{equation}\label{ms36}
 2\overline{g}(h(X,Y),E)=\overline{g}(Y,\overline{\phi}h(X,\overline{\phi}E))+\overline{g}(X,\overline{\phi}h(Y,\overline{\phi}E)).
\end{equation}
On the other hand, for $X,Y\in\Gamma(D_{0})$, we get 
\begin{equation}\label{ms37}
 2\overline{g}(h(X,Y),E)=-\overline{g}(\overline{\phi}Y,\nabla_{X}\overline{\phi}E)-\overline{g}(\overline{\phi}X,\nabla_{Y}\overline{\phi}E).
\end{equation}
It is easy to see from (\ref{ms36}) and (\ref{ms36}) that if $\overline{\phi}h(X,\overline{\phi}E)\notin\Gamma(l\mathrm{tr}(TM))$ and $\nabla_{X}\overline{\phi}E\notin\Gamma(D_{0})$, then $\overline{g}(h(X,Y),E)=0$. The other assertions follows in the same way. The converse is obvious.
\end{proof}
\begin{corollary}
 Let $(M,g,S(TM),S(TM^\perp))$ be a proper ascreen QGCR-null submanifold of an indefinite nearly cosymplectic manifold $(\overline{M}, \overline{g})$. If $M$ is $D$-totally geodesic then $\nabla_{X}^{*} \overline{\phi} E$, $\nabla_{X}^{*} \overline{\phi} W\notin\Gamma(D_{0})$, for all  $X\in\Gamma(D)$, $E\in\Gamma(D_{2})$ and $W\in\Gamma(\mathcal{S})$.
\end{corollary} 
 \begin{corollary}
 Let $(M,g,S(TM),S(TM^\perp))$ be a proper ascreen QGCR-null submanifold of an indefinite nearly cosymplectic manifold $(\overline{M}, \overline{g})$. If $M$ is $D$-totally geodesic, then $D$ defines a totally geodesic folliation in $M$.
\end{corollary}
 
\section*{Acknowledgments}
S. Ssekajja extends his sincere gratitude to the African Institute of Mathematical Sciences (AIMS) and the Simon Foundation through the RGSM-Network project, for their financial support during this research.

\end{document}